\documentclass[a4paper, oneside]{article}

\usepackage[british]{babel}
\usepackage{csquotes}
\usepackage[colorlinks=true]{hyperref}
\usepackage{pdfpages,stmaryrd,amsmath,amssymb,xcolor,tikz}

\title{Computable Ergodic Optimisation}
\author{Gayral Léo, Hoyrup Mathieu}
\date{}

\usepackage[a4paper]{geometry}

\usepackage{float}
\usepackage{caption}

\usepackage[backend=biber,style=alphabetic,maxbibnames=9,language=british,useprefix=false]{biblatex}

\DeclareFieldFormat*{title}{\textit{#1}}

\DeclareFieldFormat{doi}{\mbox{\href{https://www.doi.org/#1}{\nolinkurl{#1}}}}

\DeclareFieldFormat{url}{\mbox{\href{https://www.#1}{\nolinkurl{#1}}}}

\DeclareFieldAlias{eprint:hal}{hal}
\DeclareFieldFormat{hal}{
\mbox{\href{https://hal.archives-ouvertes.fr/#1}{\nolinkurl{#1}}}}

\DeclareFieldAlias{eprint:arXiv}{arXiv}
\DeclareFieldFormat{arXiv}{
\mbox{\href{https://arxiv.org/abs/#1}{\nolinkurl{arXiv:#1}}}}

\renewbibmacro*{doi+eprint+url}{\usebibmacro{eprint}
\iffieldundef{eprint}{\printfield{doi}\iffieldundef{doi}{\usebibmacro{url}}{}}
}

\usepackage{amsthm}
\theoremstyle{plain}
\newtheorem{theorem}{Theorem}[section]
\newtheorem{proposition}[theorem]{Proposition}

\newtheorem{corollary}[theorem]{Corollary}
\newtheorem{definition}[theorem]{Definition}
\newtheorem{remark}[theorem]{Remark}

\newcommand{\one}{\raisebox{0.9pt}{\textcircled{\raisebox{-0.9pt} {1}}}\xspace}
\newcommand{\two}{\raisebox{0.9pt}{\textcircled{\raisebox{-0.9pt} {2}}}\xspace}

\renewcommand{\epsilon}{\varepsilon}
\renewcommand{\phi}{\varphi}

\newcommand{\dom}{\mathrm{dom}}
\newcommand{\diam}{\mathrm{diam}}
\newcommand{\supp}{\mathrm{supp}}
\newcommand{\Mmax}{\M_\mathrm{max}}
\newcommand{\Lip}{\mathrm{Lip}}
\newcommand{\integral}[2]{\int #1\,\mathrm{d}#2}

\newcommand{\tuple}[1]{\left\langle #1\right\rangle}
\newcommand{\set}[1]{\left\{ #1\right\}}
\newcommand{\abs}[1]{\left| #1\right|}
\newcommand{\st}{\, :\, }

\usepackage{amssymb}
\newcommand{\A}{\mathcal{A}}

\newcommand{\F}{\mathcal{F}}
\newcommand{\G}{\mathcal{G}}
\newcommand{\M}{\mathcal{M}}
\newcommand{\Cont}{\mathcal{C}^0}

\usepackage{xspace}
\newcommand{\ie}{\emph{i.e.}\@\xspace}
\newcommand{\eg}{\emph{e.g.}\@\xspace}
\newcommand{\ce}{c.e.\@\xspace}

\usepackage{dsfont}
\newcommand{\N}{\mathds{N}}
\newcommand{\Z}{\mathds{Z}}
\newcommand{\Q}{\mathds{Q}}
\newcommand{\R}{\mathds{R}}

\begin{document}

\maketitle

\begin{abstract}
Links between physicals systems and computability properties have been an active field of investigation in recent years. Inspired by a previous work \cite{GaySabTaa25} in the context of positive temperature Gibbs measures, we prove here that in the context of zero-temperature ergodic optimisation, for a computable potential and provided with several reasonable assumptions, the maximum ergodic average is a computable real number, and the set of maximising measures is a $\Pi_1$-computable compact set.

Then, in the more specific context of symbolic dynamics, with finite-range interactions on subshifts of finite type, we provide an explicit algorithm to compute both the maximum ergodic average and the set of maximising measures in finite time, with a matching code repository \cite{Gay26}.
\end{abstract}

\section{Introduction}

Ergodic optimisation can be interpreted as the study of the zero-temperature limit behaviours of a thermodynamic model, in statistical physics. These theoretical arguments help us reach a better understanding of what actually happens in a physical material at low temperatures, notably regarding the emergence of (quasi)crystalline (a)periodic structures. We redirect interested readers to Jenkinson's relatively recent survey \cite{Jen19} for more details and context on many typical results on ergodic optimisation, as we will here focus on a quite barebones general definition, to better highlight the seldom studied computational aspects.

We start from a \emph{dynamical system} $(X,d,T)$, where $(X,d)$ is a compact metric space, and $T:X\to X$ a continuous transformation. Let $\M_T$ the set of \emph{$T$-invariant} probability measures (\ie such that $\mu\circ T^{-1} =\mu$), compact in the weak-$*$ topology. Given a continuous \emph{potential} function $\phi:X\to\R$, we can thus define the \emph{maximum ergodic average} $\beta(\phi)=\max_{\mu\in\M_T}\integral{\phi}{\mu}$ among $T$-invariant probability measures, and the (convex and compact) set of such \emph{$\phi$-maximising measures} $\Mmax(\phi)=\set{\mu\in\M_T\st\integral{\phi}{\mu}=\beta(\phi)}$.

Computable analysis aims at providing a general framework and tools to computationally describe analytic objects from mathematics, such as real numbers or closed subsets for instance. 
For general considerations on computable analysis, we refer the readers to the classical textbook by Weihrauch \cite{Wei00}, as well as more recent works for more details on links with dynamical systems \cite{Col20} and measure and probability theory \cite{Hoy21}. Computability theory has permitted the characterisation of the computably realisable values of many conjugacy invariants \cite{HoMe10,Meyer11,JeanVa14,BoDePoSaTh15,West17,HeSa18,EsNuTor23} of dynamical systems, providing practical tools to distinguish non-equivalent explicitely defined systems. Beyond these invariants, broader links between computability and dynamical systems represent an active area of research \cite{BueGraZho11, BurWol24, CotRez24, GraZho24}.

A notable fact, that will represent an unavoidable limitation for broad general results about ergodic optimisation, is that there exist some ``simple'' computable dynamical systems for which the invariant measures are known to be \emph{non-computable} \cite{GalHoyRoj11,BraGriRo12} (and that's before taking into account any further information regarding the potential, for ergodic optimisation).

A few existing results concern more specifically the interactions between statistical physics and computability, such as the computability of $\beta\mapsto\sup_{\mu\in\M_T} h(\mu)-\beta\integral{\phi}{\mu}$ (the \emph{pressure} function) on some classes of one-dimensional subshifts \cite{Spandl08,BuDaWoYa22} and higher-dimensional full shifts \cite{GaySabTaa25}, or of some thermodynamic invariants like the residual entropy \cite{BurWol20}. The general framework for computable analysis, and the specific case of computable dynamical systems will be properly introduced in Section \ref{sec:computability}.

In the context of ergodic optimisation, the space of potentials we consider (continuous, Lipschitz, Hölder, finite-range\dots) has a huge influence on the typical (\ie generic) behaviour of the model.
For instance, in the general case of continuous potentials, we have many properties that hold generically (uniqueness of the maximising measure, which has full support, zero entropy, weak but \emph{not} strong mixing...) \cite{Jen06,Bré08,Mor10,EntMiȩ20}. In comparison, if the map $T$ is \emph{expanding} (such as the shift on a one-sided space $\A^\N$, or the multiplication $\times 2$ on the circle $\R/\Z$), generically for Lipschitz potentials, the maximising measures have a \emph{periodic} support instead \cite{Con16}.

In practice, a lot of the models we consider, those we define in an explicit way, prove to be computable. Moreover, regardless, it is a prerequisite for us to be able to do any kind of numeric simulation. Hence, it makes sense to restrict ourselves to the class of \emph{computable potentials} (which is notably \emph{not} a real vector space, as not all real numbers are computable) and look at what computable properties follow. We might thus assume that $\phi$ (\ie the ``input'' of the problem) is computable, and study the complexity of $\beta(\phi)$ and $\Mmax(\phi)$ (\ie the ``outputs'' of the problem). Using mostly folkloric computability tools, in Section \ref{sec:complexity-bounds}, we identify several relevant properties for a \emph{computable} dynamical system such that the maximum ergodic average $\beta(\phi)$ is computable, with a $\Pi_1$ upper bound on the complexity of $\Mmax(\phi)$. These relevant properties are then studied on a few representative examples in Section \ref{sec:examples}.

Even though we have guarantee $\beta$ is computable under these hypotheses, the corresponding algorithm is unsatisfying, as it informally requires to enumerate $\N$ and wait until \emph{something} happens (at an unknown point in time, simply known to be finite). Hence, in Section \ref{sec:SFT}, in the more specific case of symbolic dynamics, on a subshift of finite type $X_\F\subset\A^\N$ (explicitly defined by a set of forbidden patterns $\F$), with a finite-range potential $\phi:\A^r\to\R$, we explicit an algorithm computing new forbidden patterns $\F(\phi)\supset\F$ such that $\Mmax(\phi)=\M_T\left(X_{\F(\phi)}\right)$.

\section{Background on Computability}\label{sec:computability}

In this section, we introduce from scratch computability theory while trying to stay as far from the technical details as reasonably possible, in order to define all the necessary notions for our results. Most of the notions in this section are well-documented in textbooks \cite{Wei00,BraPre03}, so we will only provide references for more specific results.

\subsection{Computability on the Natural Numbers}

There are many equivalent notions of computability (Turing machines, lambda calculus, general recursive functions\dots). For the sake of simplicity, we will here invoke a very high level view of the notion, with computable functions seen as those represented by an \emph{algorithm} (using the usual \emph{if/else} statements and \emph{for/while} loops), provided with a natural integer input, and expecting a natural integer output.

Notably, a computable function $f:\dom(f)\to\N$ can be partial, with $x\in\dom(f)$ if and only if the corresponding algorithm \emph{halts} when given $x\in\N$ as an input, in which case it outputs $f(x)\in\N$. We say that $f$ is total when $\dom(f)=\N$.

We fix an explicit bijection $\tuple{.}:\N^*\to\N$, that encodes any tuple as an integer. This bijection (and other explicit ones such as $\N\cong\Z$ or $\N\cong\Q$) allows us to transpose all the computability notions from $\N$ to $\N^*$ and more broadly to any reasonable countable set.

A set $X\subset\N$ is said to be \textbf{computably enumerable (\ce)} if it is the image set $f(\N)$ of a computable (partial) function $f$. Equivalently, $X$ is \ce if and only if it is the domain of a computable function. A set $X\subseteq\N$ is \textbf{computable} if its characteristic function $\mathds{1}_X:\N\to\set{0,1}$ is computable \emph{and total}. Equivalently, $X$ is computable if and only if both $X$ and its complement $\N\setminus X$ are \ce One of the fundamental results from computability theory is the existence of \ce sets that are not computable (this fact was already palpable in Turing's seminal work \cite{Tur37}, though the question of who actually proved/stated what is much more nuanced \cite{HamNen26}).

We can broaden these notions into a hierarchy of decision problems, with $X\subset\N$ being a $\Sigma^0_k$ set if it satisfies $x\in X\Leftrightarrow\exists n_1\in\N,\forall n_2\in\N,\exists n_3\in\N,\dots,\left(x,n_1,\dots,n_k\right)\in A$ with $k$ alternating $\exists\forall$ quantifiers over $\N$ and $A\subseteq\N^{k+1}$ a computable set. The \ce sets are precisely the $\Sigma^0_1$ sets.

Complementarily, $X$ is $\Pi^0_k$ if the leftmost quantifier is a $\forall$ instead of a $\exists$ (or, equivalently, if the complementary set $\N\setminus X$ is $\Sigma^0_k$).

\subsection{Computability on the Real Numbers}

Now, most topological spaces used in analysis, such as real numbers, are \emph{not} countable. Hence, we need to broaden the previous notions of standard computability in order to properly do computable analysis.

\begin{definition}[Hierarchy of Computable Real Numbers]\label{def:real}
A real number is \textbf{$\Sigma_1$-computable} if it is the supremum of a computable sequence of rational numbers (\ie a total computable function $f:\N\to\Q$).
Likewise, a real number is $\Pi_1$-computable if it is the infimum of such a sequence.

A \textbf{$\Delta_1$-computable} real number $x$ (simply called \textbf{computable} from now on) is both $\Sigma_1$ and $\Pi_1$. Equivalently, there is a total computable function $f:\N\to\Q$ that gives arbitrarily good approximations of $x$, such that $\abs{x-f(n)}\leq\frac{1}{2^n}$.

We can extend these notions into a broader hierarchy, where $x\in\R$ is $\Sigma_k$ if there is a computable $f:\N^k\to\Q$ such that $x =\sup_{n_1\in\N}\inf_{n_2\in\N}\dots f\left(n_1,n_2,\dots,n_k\right)$ with $k$ alternating $\sup$ and $\inf$ (resp. $\Pi_k$ if the leftmost one is $\inf$).
\end{definition}

This particular notation is not standard in computable analysis, where~$\Sigma_1$ real numbers are usually called left-computably enumerable or lower semi-computable. Still, the notation we use appears in a few articles and has the advantage of reflecting a close relationship between the hierachy of real numbers and the hierarchy we obtained before for subsets of $\N$, with $\sup$ instead of $\forall$ and $\inf$ instead of $\exists$. This correspondence is actually explicit, through the definition of real numbers as rational cuts \cite[Lemma 4.1.18]{Wei00}:
\begin{proposition}\label{prop:real-Pi_k}
A real number $x$ is $\Sigma_k$-computable if and only if the set $\set{q\in\Q\st q<x}\subset\Q$ is $\Sigma^0_k$. A real number~$x$ is~$\Pi_k$ if and only if the set~$\set{q\in\Q\st q>x}\subset\Q$ is~$\Sigma^0_k$.
\end{proposition}

Broadly speaking, if a computability notion has been defined for objects in a set $X$, witnessed by the existence of a computable function $f:\N\to\N$, then a countable family $\left(x_i\right)_{i\in\N}$ of such objects is said to be \textbf{uniformly computable} if there is a single computable function $f:\N^2\to\N$ such that for each $i\in\N$, $f(i,.):\N\to\N$ witnesses the computability of $x_i$. It is stronger than requiring each $x_i$ to be computable individually.

For example, any real number $x\in[0,1]$ can be expressed as $x=\sum_{i\in\N}\frac{b_i}{2^n}$ with each $b_i\in\set{0,1}$ being trivially computable in $\N$, but the sequence $b:\N\to\set{0,1}$ being uniformly computable directly implies $x$ itself is a computable real number (which is not the case for most real numbers, for all but a countable amount of them).

\subsection{Computable Metric Spaces}
We now explain how more general abstract spaces can be endowed with a computability structure.

\begin{definition}
A \textbf{computable metric space} is a triplet $(X,d,S)$ where $(X,d)$ is a separable metric space and $S=\set{s_i\st i\in\N}$ is a dense subset indexed by $\N$, such that the real numbers $\left(d\left(s_i,s_j\right)\right)_{i,j\in\N}$ are uniformly computable.
\end{definition}

To be precise, the uniformly computable condition means that there is a total computable function $\delta :\N^3\to\Q^+$ such that $\delta(i,j,n)$ is a $\frac{1}{2^n}$-approximation of $d\left(s_i,s_j\right)$.

In a computable metric space $(X,d,S)$, a point $x\in X$ is \textbf{computable} if there exists a computable sequence $(i_n)_{n\in\N}$ such that $d\left(s_{i_n},x\right)<2^{-n}$ for all $n$. For instance, in $(\R,d,\Q)$ where $d$ is the Euclidean metric and $\Q$ comes with a natural indexing, a real number is computable in this sense if and only if it is computable in the sense of Definition \ref{def:real}.

In such a space, a \textbf{rational ball} is an open ball $B\left(s_i,r\right)=\set{x\in X\st d\left(s_i,x\right)<r}$ where $r$ is a positive rational number. From a computable enumeration $(r_j)_{j\in\N}$ of the positive rational numbers, we derive a computable enumeration $(B_k)_{k\in\N}$ of the rational balls: $B_k=B(s_i,r_j)$ where $k=\tuple{i,j}$. We say that $U\subseteq X$ is an \textbf{effective open set} if there exists a \ce set $E\subseteq\N$ such that $U=\bigcup_{k\in E}B_k$. Computability of points can be reformulated using rational balls as follows: a point $x\in X$ is computable if and only if the set $\set{k\in\N\st x\in B_k}$ is \ce (this alternative definition can be used to define computable points in topological spaces that are not metrizable, but still countably-based).

An analog of the computability hiearchy of subsets of $\N$ can be defined for closed subsets of a computable metric space.
\begin{definition}
Let $(X,d,S)$ be a computable metric space and $k\geq 1$. A closed set $F\subseteq X$ is $\Pi_k$ if there exists a $\Sigma^0_k$ set $E\subseteq\N$ such that $X\setminus F=\bigcup_{i\in E}B_i$.
\end{definition}

In particular, a closed set is $\Pi_1$ if its complement is an effective open set. One could similarly define the $\Sigma_k$ closed sets, but it turns out that they are much less useful.

\begin{remark}
The notation $\Pi_k$ for closed sets should \emph{not} be confused with the standard notion of general $\Pi^0_k$ sets from effective descriptive set theory, which are defined from the effective open sets and using alternating countable intersections and unions rather than $\forall\exists$ quantifiers (we do not use these notions here, so we will not define them properly).

It is true that every $\Pi_k$ set is $\Pi^0_k$, but conversely $\Pi^0_k$ sets are not necessarily closed, and even then they are not $\Pi_k$ in general (unless $k=1$).
 
In the standard terminology from computable analysis, a $\Pi_k$ set is a set that is $\Pi^0_1$ with oracle $\emptyset^{(k-1)}$, where $\emptyset^{(0)}$ is the empty set and $\emptyset^{(n+1)}$ is inductively defined as the halting problem for Turing machines having an access to $\emptyset^{(n)}$ as an oracle.
Similarly, as mentioned earlier the $\Sigma_k$ and $\Pi_k$ real numbers are usually called left-\ce~and right-\ce~with oracle $\emptyset^{(k-1)}$ respectively.
\end{remark}

In Proposition \ref{prop:real-Pi_k}, we saw that the hierarchy of real numbers can be reformulated using the hierarchy of subsets of $\Q$. It is also closely related to the hierarchy of closed subsets of $\R$.
\begin{proposition}\label{prop:real-closed}
A real number $x$ is $\Pi_k$ if and only if the closed set $(-\infty,x]$ is $\Pi_k$ in $\R$.
Symmetrically, a real number $x$ is $\Sigma_k$ if and only if the closed set $[x,+\infty)$ is $\Pi_k$ in $\R$.
\end{proposition}

\begin{definition}[Computable Function]\label{def:computable-function}
Let $\left(X,d_X,S_X\right)$ and $\left(Y,d_Y,S_Y\right)$ be two computable metric spaces. A function $f:X\to Y$ is \textbf{computable} if the sets $\left(f^{-1}\left(B^Y_k\right)\right)_{k\in\N}$ are uniformly effective open sets, \ie if there exists a \ce set $E\subseteq\N^2$ such that, for every $k\in\N$, we have $f^{-1}(B^Y_k)=\bigcup_{(k,l)\in E}B_l^X$.
\end{definition}

As the preimage of each rational ball $B_k^Y$ is an open set, and these balls form a basis of the topology on $Y$, it follows that a computable function is continuous.

There are several equivalent definitions of computable functions, expressing the existence of an algorithm that, fed with arbitrary approximations of $x$, produces arbitrary approximations of $f(x)$. In particular, if $f:X\to Y$ is a computable function, then for any computable point $x\in X$, $f(x)\in Y$ is a computable point as well (and this second property is slightly weaker than the computability of $f$).

It turns out that, thus defined, computable functions interact with the effective versions of classical topological notions (open sets, compact sets\dots) in the same way continuous functions interact with those topological notions (we will see examples below). In Definition \ref{def:computable-function}, we formulated computability as the effective version of continuity notably because it makes these relationships much more straightforward.

If $f:X\to Y$ is a computable function and $P\subseteq Y$ is $\Pi_k$, then its preimage $f^{-1}(P)$ is also $\Pi_k$. However, in general $\Pi_k$ sets are not preserved by taking direct images, in particular because the image of a closed set by a continuous function is not always closed. However, we will see below that such a result holds under a computable compactness assumption about $X$ (just like, for a continuous function, the image of a compact set is a compact set).

\subsection{Computability of Compact and Closed Sets}

\begin{definition}
Let $(X,d,S)$ be a computable metric space. A compact subset $K\subseteq X$ is \textbf{computably compact} if its covering set $\set{\tuple{k_0,\ldots,k_n}\in\N^*\st K\subseteq B_{k_0}\cup\ldots\cup B_{k_n}}$ is \ce
\end{definition}

We say that the space $X$ is computably compact if, as a subset of itself, it is a computably compact set. Notably, for a computably compact space $X$, given any positive rational radius $r$, we can compute in finite time a covering of $X$ by rational balls of radius $r$ (it is known to exist, so we only need to let the enumeration run long enough to encounter one). In a computably compact space, a set is computably compact if and only if it is $\Pi_1$. The image of a computably compact set by a computable function is computably compact.

When the space is compact but not computably so, the complexity of the covering set of a $\Pi_1$ closed set (which is necessarily compact itself) may increase, but only from $\Sigma^0_1$ ( \ce) to $\Sigma^0_2$.

\begin{proposition}\label{prop:compact}
If $(X,d,S)$ is compact, then the covering set of any $\Pi_1$ closed set is $\Sigma^0_2$.

\begin{proof}
Let $K$ a $\Pi_1$ closed set (\ie $X\setminus K$ is effectively open). Let $\left(l_i\right)_{i\in\N}$ be a computable sequence of natural numbers such that $X\setminus K=\bigcup_{i\in\N}B_{l_i}$. For a rational ball $B_k=B(s,r)$ and a rational factor $a\in (0,1)$, let $\overline{B}^a_k=\overline{B}(s,ar)=\set{x\in X\st d(s,x)\leq ar}$. One has the following equivalences:
\begin{align*}
K\subseteq B_{k_0}\cup\ldots\cup B_{k_n}&\Longleftrightarrow X\subseteq B_{k_0}\cup\ldots\cup B_{k_n}\cup\bigcup_{i\in\N}B_{l_i} &(\text{as }X=K\sqcup X\setminus K)\\
&\Longleftrightarrow\exists p, X\subseteq B_{k_0}\cup\ldots\cup B_{k_n}\cup\bigcup_{i\leq p}B_{l_i}&(\text{by compactness})\\
&\Longleftrightarrow\exists p,a, X\subseteq\overline{B}^a_{k_0}\cup\ldots\cup\overline{B}^a_{k_n}\cup\bigcup_{i\leq p}\overline{B}^a_{l_i}&(\text{by compactness})\\
&\Longleftrightarrow\exists p,a,\forall i, s_i\in\overline{B}^a_{k_0}\cup\ldots\cup\overline{B}^a_{k_n}\cup\bigcup_{i\leq p}\overline{B}^a_{l_i}\, .&(\text{by density of }\left(s_i\right)_{i\in\N})
\end{align*}
Observe now that, as the real numbers $d\left(s_i,s_j\right)$ are uniformly computable, the predicate $s_i\in\overline{B}^a_k$ is $\Pi_1^0$ (\ie it can be replaced by some formula $\forall t,\, g(a,i,k,t)$ with $g$ a computable algorithm). It follows that the predicate in the previous equivalence is itself $\Sigma^0_2$.
\end{proof}
\end{proposition}

Hierarchies of subsets of $\N$ and of real numbers are symmetric, in the sense that each class $\Pi^0_k$ or $\Pi^k$ comes with its dual notion $\Sigma^0_k$ or $\Sigma_k$.

In practice, it turns out that the usuful ``dual'' notion of $\Pi_k$ for closed sets is not, as one would naturally expect, that of a $\Sigma_k$ closed set (nor an analogous definition for open sets). Intuitively, a closed set is $\Pi_1$ if its complement can be enumerated as a union of rational balls. It does not make sense to enumerate the closed set itself, which is usually uncountable, but the next notion turns out to be a useful way of describing the set ``from inside''.

\begin{definition}
Let $(X,d,S)$ be a computable metric space $(X,d,S)$. A closed set $A\subseteq X$ is \textbf{computably overt} if the set $\set{k\st A\cap B_k\neq\emptyset}$ is \ce
\end{definition}

Under the assumption that the metric is complete, computable overtness can be characterized in a particularly simple way \cite[Theorem 3.8]{BraPre03}.
\begin{proposition}\label{prop:overt}
When the metric is complete, a closed set is computably overt if and only if it contains a uniformly computable dense sequence.
\end{proposition}

The next result summarises the computable complexity of maximising a function over a compact set. These results appear in many places in the literature on computable analysis, but we include a sketch of proof for clarity.
\begin{proposition}\label{prop:max}
Let $K\Subset\R$ be compact:
\begin{itemize}
\item If $K$ is computably compact, then $\max K$ is $\Pi_1$,
\item If $K$ is computably overt, then $\max K$ is $\Sigma_1$.
\end{itemize}

Let $(X,d,S)$ a computable metric space, $f:X\to\R$ computable and $P\subseteq X$ compact.
\begin{itemize}
\item If $P$ is $\Pi_1$, then $\max f$ is $\Pi_2$.
\item If $P$ is computably compact, then $\max f$ is $\Pi_1$,
\item If $P$ is computably overt, then $\max f$ is $\Sigma_1$,
\item If $P$ is both computably compact and overt, then $\max f$ is $\Delta_1$, i.e.~computable.
\end{itemize}
\end{proposition}
\begin{proof}
One needs to show that the set $\set{q\in\Q\st\max K<q}$ is \ce to prove the first item. As $K$ is compact (thus bounded), we have $\max K<q$ if and only if there exists a rational number $p<q$ such that $K\subseteq (p,q)$. This relation is \ce when $K$ is computably compact, so $\max K$ is $\Pi_1$ per Proposition \ref{prop:real-closed}.

If $K$ is computably overt then it contains a dense computable sequence by Proposition \ref{prop:overt}, its supremum is $\Sigma_1$ by definition, and that supremum is precisely $\max K$.

Now, if $P$ is $\Pi_1$, then its covering set is $\Sigma^0_2$ by Proposition \ref{prop:compact}, and we must show that the covering set of $f(P)$ is also $\Sigma^0_2$. Let $E\subseteq\N^2$ be a \ce set such that $f^{-1}(B^\R_k)=\bigcup_{(k,l)\in E}B^X_l$. The sets $E_k=\set{l\st (k,l)\in E}$ are uniformly \ce By compactness of $P$, one has:
\begin{align*}
f(P)\subseteq B^\R_{k_0}\cup\ldots\cup B^\R_{k_n}&\Longleftrightarrow P\subseteq\bigcup_{l\in E_{k_0}}B_l^X\cup\dots\cup\bigcup_{l\in E_{k_n}} B_l^X\\
&\Longleftrightarrow\exists l_0,\ldots,l_p\in\bigcup_{0\leq i\leq n} E_{k_i}\text{ such that }P\subseteq\bigcup_{0\leq j\leq p}B^X_{l_j}\, ,
\end{align*}
which is $\Sigma^0_2$ if the covering set of $P$ is. Therefore, $\max f(P)<q$ if and only if $\exists p, f(P)\subseteq (p,q)$ which is $\Sigma^0_2$, so $\max f(P)$ is $\Pi_2$.

The remaining statements follow from the first two assertions, using the fact that taking the image of a set by a computable function preserves computable compactness and computable overtness.
\end{proof}

\subsection{Space of Probability Measures}

When $X$ is compact metrisable, the space $\M(X)$ of Borel probability measures over $X$ endowed with the weak-$*$ topology is metrisable, and a compatible metric is given by the Wasserstein metric
\[
W_1(\mu,\nu)=\sup_{f\in\Lip}\abs{\integral{f}{\mu}-\integral{f}{\nu}}
\]
where $\Lip$ is the set of $1$-Lipschitz functions $f:X\to\R$ \cite[Section 7.1]{AmbGigSav05}. The separability of $X$ implies the separability of $\M(X)$, where a dense sequence of measures is given by the set $D$ of finite convex combinations with rational coefficients of Dirac measures supported on a countable dense subset of $X$. The compactness of $X$ also implies the compactness of $\M(X)$ respectively. These implications also hold in a computable setting (see \cite[Proposition 4.1.3]{HoyRoj09} and \cite[Lemma 2.12]{GalHoyRoj11}).

\begin{proposition}
If $(X,d,S)$ is a computable metric space, then $(\M(X),W_1,D)$ is also a computable metric space. If $X$ is moreover computably compact, then so is $\M(X)$.
\end{proposition}

The computable structure on $\M(X)$ makes integration of computable functions a computable operator \cite[Corollary 4.3.2]{GalHoyRoj11}.
\begin{proposition}\label{prop:computable-integral}
If $\varphi:X\to\R$ is computable, then the function $\phi^*:\mu\in\M(X)\mapsto\integral{\phi}{\mu}\in\R$ is also computable.

When $X$ is computably compact, this result is uniform in $\phi$, otherwise one needs to know an upper bound on $\abs{\phi}$ for each $\phi$, and encode it in the corresponding algorithm.
\end{proposition}

If $T:X\to X$ is Borel measurable, then let $\M_T(X)$ be the set of $T$-invariant Borel probability measures over $X$.

\begin{proposition}\label{prop:invariant-complexity}
Let $(X,d,S)$ be a computable metric space and $T:X\to X$ be computable. The set $\M_T(X)$ is a $\Pi_1$ subset of $\M(X)$.
\end{proposition}

The set $\M_T(X)$ is not computably overt in general. It was shown in \cite{GalHoyRoj11} that there exists a computable function from the circle to itself that has no computable invariant measure, so in that case $\M_T(X)$ does not contain any computable point.

\section{General Theoretical Upper Bounds}\label{sec:complexity-bounds}

\begin{definition}[Computable Dynamical System]
A computable dynamical system $(X,d,S,T)$ is, expectedly, a computable metric space $(X,d,S)$ provided with a computable transformation $T:X\to X$.
\end{definition}

In the context of \emph{ergodic optimisation}, our default assumption here is that $X$ is compact (but \emph{not necessarily} computably so) and that the potential $\phi:X\to\R$ is computable. Consider the following hypotheses:
\begin{itemize}
\item[ \one] $X$ (and thus incidentally $\M(X)$ and $\M_T(X)$) is computably compact.
\item[ \two] $\M_T(X)$ is computably overt.
\end{itemize}

As explained in a general setting in the previous section, \one means that the compactness of $X$ can be expressed (thus used) computationally. In particular, $\M_T(X)$ can be described ``from the outside'' by eleminating non-intersecting balls. This property usually follows from the fact that the space $X$ has a ``nice'' and explicit definition.

Hypothese \two is more complex to obtain (see Section \ref{sec:examples} for a conversation on this topic in the context of ergodic optimisation) and implies that $\M_T(X)$ can itself be seen as a computable space, computably endowed with its own countable dense family.

\begin{theorem}\label{thm:upper-bound}
Let $(X,d,S,T)$ be a computable dynamical system, with $X$ compact, and let $\phi:X\to\R$ a computable potential. The following complexity upper bounds hold:

\begin{tabular}{|c|c|c|}\hline
Hypotheses & $\beta(\phi)$ in $\R$ & $\Mmax(\phi)$ as a compact in $\M(X)$\\\hline
General Case & $\Pi_2$ & $\Pi_3$\\\hline
 \one ($X$ is comput.\@ compact) & $\Pi_1$ & $\Pi_2$\\\hline
 \two ($\M_T(X)$ is comput.\@ overt) & $\Sigma_1$ & $\Pi_1$\\\hline
 \one{+}\two & $\Delta_1$ & $\Pi_1$\\\hline
\end{tabular}

\begin{proof}
As $\phi$ is computable, $\phi^*:\mu\in\M(X)\mapsto\integral{\phi}{\mu}$ is itself computable by Proposition \ref{prop:computable-integral}.

The real number $\beta(\phi)$ is thus the maximum of $\phi^*$ on the subset $\M_T(X)$ (which is $\Pi_1$ from Proposition \ref{prop:invariant-complexity}), so it is $\Pi_2$ (from Proposition \ref{prop:max}). The other bounds on $\beta(\phi)$ follow by likewise applying Proposition \ref{prop:max} when $\M_T(X)$ is computably compact and/or overt.

Now, as $\M_T(X)$ itself is a $\Pi_1$ set (Proposition \ref{prop:invariant-complexity}), it suffices to obtain a $\Pi_k$ bound on $\set{\mu\in\M(X)\st\integral{\phi}{\mu}\geq\beta(\phi)}=\left.\phi^*\right.^{-1}([\beta(\phi),\infty))$ for it to apply to $\Mmax(\phi)$ too. As $\phi^*$ is computable, the preimage of a $\Pi_k$ closed set is itself a $\Pi_k$ closed set of $\M(X)$. To obtain a bound, we thus simply need to find $k$ such that $[\beta(\phi),\infty)$ is a $\Pi_k$ subset of $\R$, \ie $\beta(\phi)$ is a $\Sigma_k$ real number (see Proposition \ref{prop:real-closed}). Noticing $\Pi_k\subset\Sigma_{k+1}$ (we simply add a leftmost useless $\inf$), the result follows from the first part of the theorem.
\end{proof}
\end{theorem}

\section{Brief Survey of the Complexity of Dynamical Systems}\label{sec:examples}

To ensure our study and hierarchy in Theorem \ref{thm:upper-bound} is, to a degree, rooted in reality, we will in this section list the dynamical systems most representative of (the early days of) ergodic optimisation, and discuss whether they satisfy \one and/or \two, whether they are computably compact and/or overt.

This section is purposefully an informal conversation, it includes no result that we consider new, though some facts are hard-to-pinpoint folklore.

\subsection{General Context}

As already explained, \one directly follows from the definition of the space $X$ and thus is easy to satisfy on ``naturally occurring'' spaces (and it will be satisfied on all the examples we discuss), whereas \two is a convoluted function of how $T$ acts on $X$, not always satisfied and harder to establish -- positively or negatively. Still, those two notions are theoretically unrelated and complementary, and there are dynamical systems for which only \one or \two (or none) holds.

We acknowledge the focus of ergodic optimisation in the past twenty years has shifted from the compact framework we use to non-compact settings (notably by having either an unbounded real space such as $\R^d$, or a shift space $\A^G$ on a non-compact alphabet such as $\A=\N$), which fall outside of the established scope of the paper, making \two notably out of reach.

Discontinuous transformations (such at the Gauss map $T:x\mapsto\frac{1}{x}-\left\lfloor\frac{1}{x}\right\rfloor$ on the unit interval) are also much more intrinsically out of scope, by virtue of being non-computable.

Another prevalent case that escapes the scope of this paper is that of continuous flows where, instead of a computable transformation $T:X\to X$, we have a family $T=\left(T_t\right)_{t\in\R^+}$ (such that $T_u\circ T_v=T_{u+v}$ for any $u,v\geq 0$, with $(t,x)\mapsto T_t(x)$ computable). This plays no role on whether \one holds, but makes the study of \two even harder -- however, assuming \two holds for the set of measures invariant for \emph{the whole family $T$}, the complexity upper bounds in the previous section should follow likewise.

\subsection{One Dynamic in a Three-Models-Shaped Trenchcoat}

Consider $P(x)=4x(1-x)$ the Ulam-von~Neumann map on $X=[0,1]$ (with the usual metric on $\R$), the expanding map $T(\theta)=2\theta$ on $Y=\R/\Z$ (with the corresponding induced metric) and the full shift $\sigma$ on $Z=\set{0,1}^\N$ (with $d(x,y)=2^{-\inf\set{i\st x_i\neq y_i}}$). All these models were mentioned in Jenkinson's 2006 lecture notes \cite{Jen06} (also available \href{https://webspace.maths.qmul.ac.uk/o.m.jenkinson/lectures13.pdf}{on their webpage}), which makes them staple examples for ergodic optimisation.

Due to its explicit symbolic structure, the full shift $(Z,\sigma,d)$ is a computationally well-behaved space, with the dense family $\left(w0^\infty\right)_{w\in\set{0,1}^*}$. The rational balls are given by the clopen cylinders $([w])_{w\in\set{0,1}^*}$, so that the covering set is not only \ce but \emph{decidable}, \one holds trivially. For periodic configurations $\omega = w^\infty$ (with $w\in\set{0,1}^*$ a finite word of length $p=\abs{w}$), the corresponding ``periodic measures'' $\mu_\omega=\frac{1}{p}\sum_{i=0}^{p-1}\delta_{\sigma^i(\omega)}\in\M_\sigma(Z)$ represent a uniformly computable dense sequence, so that \two holds.

Consider $\Gamma:Z\to Y$ defined as $\omega\mapsto\sum_{i\in\N}\frac{\omega_i}{2^i}$.
The map is not injective, as collisions between exactly two elements occur for any words of the form $u01^\infty$ and $u10^\infty$ (with a prefix $u\in\set{0,1}^*$), as well as $1^\infty$ and $0^\infty$.
This set of words $U$ can only support exactly two $\sigma$-invariant measures, on the fixed points $0^\infty$ and $1^\infty$. Likewise, looking at the image set $\Gamma(U)$, we obtain specifically diadyc fractions, which all get sent to the fixed point $0$ after finitely many iterations. Outside of those, $\Gamma:Z\setminus U\to Y\setminus\Gamma(U)$ is a bimesurable bijection such that $\Gamma\circ\sigma=T\circ\Gamma$, establishing notably a one-to-one correspondence between the invariant measures, and computably matches the (dense) periodic sequences (\ie a computation basis of $Z$) and (dense) non-dyadic rational numbers (\ie a computation basis of $Y$).

Likewise, with $\Psi:Y\to X$ defined as $\theta\mapsto\sin^2(\pi\theta)$, we have a non-injective computable map with two branches (it computably bijects $\left[0,\frac{1}{2}\right]$ with $[0,1]$ going upwards, and then $\left[\frac{1}{2},1\right]$ with $[0,1]$ going downwards, the only one-to-one points being $\Psi(0)=0$ and $\Psi\left(\frac{1}{2}\right)=1$). Notably, $\Psi(x)=\Psi(-x)=\Psi(1-x)$. Now, $\Psi\circ T=P\circ\Psi$, so once again, they both encode in superficially different ways the same core dynamical behaviour, and the dense family of periodic orbits in $Y$ thus gives a dense family of periodic orbits in $X$.

The point of all this is that the full shift $Z$ is much easier to study computably to obtain \two, which in turn implies that \two also holds for $Y$ and $X$ in a very straightforward way, much simpler than trying to find the uniformly computable and dense periodic orbits in $X$ directly (\ie by solving the polynomial equations $P\circ\cdots\circ P(x)=x$) and identifying the structure of $\M_P(X)$ itself.

This correspondence also allows us to exhibit ``weird'' invariant measures for $(X,P)$ or $(Y,T)$ corresponding to known atypical ones on the full shift, such as the (unique) invariant measure on the strictly ergodic Thue-Morse subshift.

\subsection{Other Models}

Let us mention some more models, taken from Jenkinson's more recent survey \cite{Jen19}.

Generalising the aforementioned expansion $T_0:x\mapsto 2x$ of $Y$ (on which \one holds), there is a whole family of Manneville-Pomeau expanding circle maps of the form $T_a:x\mapsto x + x^{1+\alpha}$ (with $\alpha\geq 0$). In this case, there is still a (topologically semi-conjugate) full shift hidden inside, as is more broadly the case for \emph{any} expanding map on $Y$ of any degree $d\geq 2$ \cite[Section 7.4]{KatHas03}. Identifying this full shift means identifying the \emph{branches} of the expanding map (whose endpoints are the premiages of $0$). For instance, as $T_a$ is of degree $2$ here, it would mean computing the unique $0<\gamma<1$ such that $\gamma+\gamma^{1+\alpha}=1$, by dichotomic search for example. We can then repeat this process by computing the two pre-images of $\gamma$, on and on to obtain the $2^k$ branches of $T_a^k:=\underset{k\text{ times}}{T_a\circ\cdots\circ T_a}$. By excluding this countable set $C$ of uniformly computable points (analogous to dyadic fractions in the case of $T_0$, whose orbits eventually sink in the fixed point $0$), we can thus simply map the remaining elements $x\in Y\setminus C\mapsto\left(\mathds{1}_{(\gamma,1)}\left(T_a^k(x)\right)\right)\in\set{0,1}^\N$, establishing by design a computable correspondence between invariant measures of the binary full shift $(Z,\sigma)$ and those of any Manneville-Pomeau circle map $\left(Y,T_a\right)$, provided $a$ is itself computable (and more broadly between the full shift $\llbracket 1,d\rrbracket^\N$ and any computable expanding map of degree $d$), which guarantees \two holds.

Another interesting example is the quintessential model of spin glasses, the XY model, in which we don't just look at \emph{one} spin in $Y$, but a whole glass in $Y^\Z$, with the shift $\sigma$ as a natural transformation. In this case, periodic configurations with ``rational'' spins in $\Q/\Z$ (or rather their projection on $Y$) provide both a computably enumerable dense family of $Y^\Z$ (hence define a notion of computability) and an induced dense family of invariant measures (so that \two holds). By using the same computable compactness as for $Y$ (with finite coverings by increasingly small rational intervals $B_k$) but on increasingly large windows $\llbracket -n, n\rrbracket\subset\Z$ of spins, one can easily reach the conclusion that \one still holds (by checking, for each potential covering $\left(B_{\left<k^i\right>}\right)_i$ of $Y^\Z$ by ``rational rectangles'' $B_{<k>}:=Y^\infty\times B_{k_{-n}}\times\dots\times B_{k_n}\times Y^\infty$, if we can find a fine enough covering of $Y$, extended as a covering of $Y^{\llbracket -n,n\rrbracket}$ by small ``rational cubes'', that refines $\left(B_{\left<k^i\right>}\right)_i$).

All the examples discussed so far are deeply related to a full shift space, but this is not necessarily the case in general.

\subsection{Generalising the Ulam-von~Neumann Map}

Generalising the Ulam-von~Neumann map gives us the family of logistic maps $f_\lambda:x\mapsto\lambda x(1-x)$ for $\lambda\in [0,4]$.

These maps are part of a broader family of continuous interval transformations $T:I\to I$, for which higly chaotic behaviours can occur, notably regarding which periodic orbits can exist (\emph{Period Three Implies Chaos} \cite{TieYie75}, and more broadly the existence of $p$-periodic orbits implies the existence of $q$-periodic orbits for any $p\succcurlyeq q$ in \emph{Sharkovskii's order} \cite{Sha64,ShaTol95}).

In the case of logistic maps $f_\lambda$, from $\lambda\leq 1$ (where $0$ is the unique attractor) to $\lambda>1+\sqrt{8}$ (where $f_\lambda$ admits a $3$-periodic orbit, thus of any other period), we cross many thresholds that change the maximal period (in Sharkovskii's order). Notably, starting from $\lambda\leq 1$, we encounter a series of period-doubling thresholds (two-periodic orbits appear beyond $\lambda=3$, $4$-periodic orbits appear beyond $\lambda=1+\sqrt{6}$, accumulating (at a geometric rate known as the \emph{Feigenbaum constant}) to a computable limit \cite[Theorem 6]{Hoy07} for which no proper algebraic expression is known.

Some of these thresholds correspond to the apparition of isolated periodic points, thus represent discontinuities in the map $\lambda\mapsto\Mmax\left(f_\lambda\right)$. For the usual reasons of undecidability of equality over computable real numbers, unless we are given supplementary information on an \emph{a priori} arbitrary computable number $\lambda$, it could lie on such a problematic threshold, in which case we will be unable to properly satisfy $ \two$ (even if the threshold $\lambda_t$ itself is a computable real number and we mathematically know that $\Mmax\left(f_{\lambda_t}\right)$ is computably overt).

A likewise potential obstruction to \two can be found in the existence of computable parameters $\lambda$ for which $f_\lambda$ has a non-computable SRB measure \cite{RojYam20,RojYam24}.

\subsection{Symbolic Shift Spaces}

For simplicity's sake, we will say a shift space $X$ is a (closed) subset of $\A^\N$ (with a finite alphabet $\A$) invariant under the shift action $\sigma$.
Symbolic shift spaces are ``everywhere'', beyond the works cited by Jenkinson's (extensive) surveys, for many reasons we will not detail. They will become our focal point in the following final section of this article.

As seen above a full shift can, with its very straightforward dynamics, deobfuscate the complex behaviour behind analytic systems such as the Ulam-von~Neumann map $x\mapsto4x(1-x)$ on the unit interval.

Given a set of finite words $\F\subset\A^*$ called \emph{forbidden patterns}, one can define the induced subshift $X_\F$ made of configurations avoiding $\F$. For any subshift $X$, denote $L(X)\subset\A^*$ its \emph{language} (each word in $L(X)$ is the finite prefix of an element of $X$), then the countable set $\F=\A^*\setminus L(X)$ is such that $X=X_\F$. 
 
This apparent structural simplicity of the full shift hides the very complex and rich behaviours that can occur on subshifts depending on the setting, whether we are looking at a one-sided or two-sided shift (in $\A^\Z$), whether it is a \emph{Subshift of Finite Type} (an SFT, such that $X=X_\F$ for a finite set $\F$), sofic (with $\F$ a regular language) or effective (with $\F$ a \ce subset of $\A^*$).

Another point of interest, which we will briefly discuss, is the case of \emph{higher-dimensional} subshifts, \ie closed subsets of $\A^{\Z^d}$ invariant under translations along every direction. Of interest here is the fact that for finite sets of forbidden patterns $\F$, \emph{unlike} the one-dimensional case where there is an algorithmic procedure to know whether $X_\F$ is non-empty (and if so, then we can computably enumerate increasingly long periodic orbits which give a dense family of invariant measures, such that \two holds), the higher-dimensional domino problem is \emph{undecidable} and there are explicit non-empty SFTs (such as the Robinson or the Kari-Čulík tiling) for which no periodic orbit exists or, even worse, where $\F$ is explicit but not a single point of $X_\F$ is a computable point of $\A^{\Z^2}$ \cite{Hanf74,Myers74} (so that even endowing it with a computable space structure would require at best careful work, before starting to discuss \one and \two). Consequently, the periodic orbits can no longer form an easily enumerable dense family to satisfy \two, as was the case in all the previous examples -- this argument might appear a bit fallacious given the Robinson tiling is uniquely ergodic (thus trivially satisfies \two) and its invariant measure is very well known and understood, but the point we are making here is that in the general two-dimensional case, it is \emph{impossible} to algorithmically obtain this information on $\M_T\left(X_\F\right)$ being given $\F$, and thus this provides a very natural class of dynamical systems \emph{simple to describe but hard to understand} for which \two is unknown in general.

Back to the one-sided one-dimensional case, in the following section, using graph arguments with a flavour similar to that of those used to decide whether $X_\F$ is non-empty and if so to find its periodic orbits, we will give algorithmic arguments to compute both $\beta(\phi)$ and $\Mmax(\phi)$ for finite-range rational potentials $\phi:\A^r\to\Q$.

\section{Finite-Time Polynomial Complexity for Markov Shifts}\label{sec:SFT}

In Theorem \ref{thm:upper-bound}, we identified a setting in which $\beta(\phi)$ is computable.

By assuming the space $X$ is computably compact (\one), we can enumerate arbitrarily accurate finite coverings of $\M_T(X)$ by rational balls, and each such covering gives us an accordingly close upper bound on $\beta(\phi)$. By assuming $\M_T(X)$ is computably overt (\two), we obtain a uniformly computable, countable, dense subset $S\subset\M_T(X)$ (\eg periodic orbits for a full shift), and getting arbitrarily good approximations of $\phi^*(\mu)$ for each $\mu\in S$ gives an arbitrarily close lower bound on $\beta(\phi)$. Putting those two together, assuming \one{+}\two, we obtain an algorithm that can approximate $\beta(\phi)$ with arbitrary precision.

The main drawback of this all-purposes approach is that we have no specific information on the structure of $(X,T)$ as a dynamical system, everything is abstracted away and obfuscated behind our computable objects. Consequently, we can provide no information on the \emph{computational complexity} of the method.

The goal of this final section is precisely to address this blind spot by considering a very explicit model both computationally simple to manipulate and mathematically well-understood, that of one-dimensional SFTs, sometimes also called \emph{Markov shifts} in the mathematical literature.

\subsection{Controlling the Maximising Measures for Well-Behaved Potentials}

Let $(X,T)$ be a compact dynamical system, with $\phi:X\to X$ a continuous potential. Intuitively, what is $\Mmax(\phi)$ expected to look like?

Assuming $\phi$ is constant on every $T$-orbit (which is a way too strong assumption), it is quite easy to see that $\beta=\max \phi$ and that $\Mmax(\phi)=\M_T\left(\phi^{-1}(\beta)\right)$. More reasonably, as $\phi$ gives a weight to every configuration, it informally induces an average weight for every orbit. By defining $Y$ as the union of all the $\phi$-maximal orbits, one can expect to have $\Mmax(\phi)=\M_T(Y)$, \ie the maximising measures are those supported by maximising orbits.

While this statement is not provably true in a general setting, it actually holds for some well-behaved dynamical systems (\eg weakly expanding like the one-sided shift $\A^\N$, or hyperbolic with weak local product like the two-sided shift $\A^\Z$ \cite{Bou01}), provided $\phi$ is regular-enough (it must satisfy the \emph{Walters condition} so that some version of \emph{Mañé’s lemma} holds).

The Walters condition on a continuous potential $\phi:X\to X$, as defined in Bousch's article \cite{Bou01}, is a sort of uniform continuity property for $\phi$ along trajectories of arbitrary lengths. More precisely, with $(X,d,T)$ a metric compact system, if the $T$-orbits of $x,y\in X$ stay $\epsilon$-close for some (arbitrary) amount $n\in\N$ of $T$-transitions, then
\[
\abs{\sum_{i=0}^n \phi\circ T^i(x)-\phi\circ T^i(y)}\leq h(\epsilon)\, ,
\]
with $h:\R^+\to\R^+$ a non-decreasing function independent of $n$, such that $\lim_0 h = 0$, called a \emph{Walters module} for $\phi$.

\begin{remark}
The notion of Walters module is independent from the computability of a potential $\phi$. Trivially, non-computable finite-range (thus Walters) potentials exist. Less obviously but still easily, we can exhibit hand-crafted non-Walters computable potentials for which $\beta(\phi)$ and $\M_\sigma(\phi)$ are explicit.
\end{remark}

Notably, in the case of a one-sided shift space $\A^\N$ (though every argument translates for two-sided shifts), with $d(x,y)=\abs{\A}^{-\min\set{i\in\N\st x_i\neq y_i}}$, the orbits of $x$ and $y$ being $\epsilon$-close for $n$ steps actually means they were $\frac{\epsilon}{\abs{\A}^i}$-close at step $n-i$. Thus, when $\phi$ has summable variations (\ie $F:\epsilon\mapsto\sup_{d(x,y)\leq\epsilon}\abs{\phi(x)-\phi(y)}$ is integrable on a neighbourhood of $0$), $\eta\mapsto\int_0^\eta F$ provides a Walters module. This argument holds in particular for finite-range potentials, which are \emph{locally constant}.

This class of interactions has been widely studied, and is heavily related to SFTs. Notably, using cohomology arguments, Bousch proved that if $\phi$ is a finite-range potential, then there is an SFT $X$ such that $\Mmax(\phi)=\M_\sigma(X)$.
What he did \emph{not} explain is how one can construct, or \emph{compute}, a set $\F$ of forbidden patterns such that $X=X_\F$. This will be our focal point from now on.

\begin{remark}
There are fundamental obstructions to an exact computation of $\F$ in finite time, due to the undecidability of equality over computable real numbers. If we consider a general computable $\phi:\A^r\to\R$, we cannot have an algorithm providing $\F$ in finite time.

Consider the following counter-example, on the binary alphabet $\A=\set{0,1}$. Let $\phi:\A^2\to\R$ such that $\phi(01)=\phi(10)=-1,\phi(00)=0$ and $\phi(11)=\lambda$ is some computable real number. Denote $X$ the SFT that supports maximising measures. Switching from $0$ to $1$ in a configuration in $\A^\N$ always results in a clear negative penalty compared to $\beta=\max(0,\lambda)$, corresponding to the fixed-point configurations $0^\infty$ and/or $1^\infty$. In this case, either $\lambda>0$ (so $X=\set{1^\infty}$), $\lambda<0$ (so $X=\set{0^\infty}$) or $\lambda=0$ (so $X=\set{0^\infty,1^\infty}$). While $\lambda\neq 0$ can be observed in finite time, the case $\lambda=0$ can only be \emph{de facto} concluded after infinitely many computations, and even providing a \emph{subset} of $X$, or providing some measure $\mu\in\Mmax(\phi)$, cannot be done in finite time for the same reasons.

Using this counterexample, we can see how impossible it is to exhibit even \emph{one} measure in $\Mmax(\phi)$. After finitely many computation steps, the most we can know mathematically about $\lambda$ is an open rational interval $I$ such that $\lambda\in I$. For a given such interval $I$, assuming $0\in I$ too, any measure supported by $\set{0^\infty,1^\infty}$ \emph{could} be in $\Mmax(\phi)$ (when $\lambda=0$), but $I$ also matches mutually exclusive cases with $\Mmax(\phi)$ being $\set{\delta_{0^\infty}}$ if $\lambda<0$, or $\set{\delta_{1^\infty}}$ if $\lambda>0$. Any measure $\mu$ \emph{could} be a wrong guess, and the interval $I$ itself could be arbitrarily small so there is no hope for some sort of approximation $I\mapsto\mu_I$ such that $d\left(\mu_I,\Mmax(\phi)\right)\to 0$ as $\diam(I)\to 0$.
This computability obstruction is intrinsic, it cannot be fixed by algorithmic means unless we are provided with an ``oracle'' telling us whether any two computable numbers are equal.

The best we could do in such a setting, adapting the ideas and algorithms that will be introduced below, is to replace each computable number used in $\phi$ by an approximate rational interval of length $\frac{1}{2^n}$, and eliminate the configurations that \emph{cannot} be in $X$, that \emph{cannot} maximise $\phi^*(\mu)$, rather than trying to find one that is. This would give us an algorithm $(\phi,n)\mapsto\F_n\subset\A^r$ such that $\Mmax(\phi)\subset\M_\sigma\left(X_{\F_n}\right)$. In this scenario, the set of forbidden patterns $\F_n$ is increasing, stationary as $\A^r$ is finite, and converges to the ``true'' $\F$ such that $X_\F$ is the support of maximising measures. This fact was already a consequence of Theorem \ref{thm:upper-bound}, as being computed ``from the outside'' by elimination is exactly what $\Pi_1$ means, and we are in a setting where \one{+}\two holds.
\end{remark}

Hence, to provide an algorithm computing $\phi\mapsto\F$ (so that $\Mmax(\phi)=\M_\sigma\left(X_\F\right)$), we will need some further assumptions on the values taken by $\phi$, and clarify what we mean here by algorithm and complexity.

\subsection{From Cohomology to Graph Algorithms}

In the following, we consider $(\mathds{F},+,\times)$ an abstract ordered subfield of $\R$. The field $\mathds{F}$ could simply be $\R$, or the subfield of all computable real numbers, but as discussed in the previous remark this would evade our ``practical computations'' approach. Thus, in the following algorithmic considerations, we assume that we can ``compute'' the sum $(x,y)\mapsto x+y$, the rational product $(q,x)\mapsto q\times x$ (with $q\in\Q\subset\mathds{F}$) and test equality $(x,y)\mapsto (x=y)$ and inequality $(x,y)\mapsto (x<y)$ for our elementary operations. This holds for $\Q$ (by encoding a rational as pair of integers), but also (in a very unpractical way) for the field of all real algebraic numbers.

\begin{remark}
The expression ``elementary operations'' is a bit deceitful here, as one would expect our individual variables in $\mathds{F}$ to take $O(1)$ space, and our operations to take $O(1)$ time, but this obviously cannot hold for an infinite set such as $\Q$.

It must be understood as already-implemented black boxes, which we will use as basic blocks of the algorithm, and which we will count to describe the time/space complexity of said algorithm. In practical implementations, we must choose between using $O(1)$ space/time objects (such as 64bit floating-point numbers) with rounding errors (breaking $\phi\mapsto\F$ around fringe cases), or a dynamically-sized structure (such as Python long integers and fractions) with perfect precision but operations whose complexity depends on the size of the input.
\end{remark}

We will consider a finite-range potential $\phi:\A^\N\to\mathds{F}$. As $\phi$ is locally constant, there is a cylinder size $r$ such that, on any cylinder $[w]$ with $w\in\A^r$, $\phi([w])$ is a singleton. We can thus equivalently see this potential as a function $\phi:\A^r\to\mathds{F}$, and this is how it will be ``encoded'' in a computer, by hardcoding its range $r$ in its definition.

A useful tool, both for the next proposition and for what follows afterwards, will be the \emph{De Bruijn} directed graph $G_n$ on words of length $n>0$, with transitions between words overlapping over $n-1$ letters, \ie $G_n:=\left(\A^n,\set{ aw\to wb\st a,b\in\A,w\in\A^{n-1}}\right)$.

\begin{proposition}
Let $\left(\A^\N,\sigma\right)$ be the one-sided full shift on $\A$. Let $\phi:\A^r\to\mathds{F}$ be a potential of finite range $r\in\N$ (so that $\phi(x):=\phi\left(x_0\dots x_{r-1}\right)$ for an actual infinite word $x\in\A^\N$), and $X$ be the (smallest) SFT such that $\Mmax(\phi)=\M_\sigma(X)$ (\ie it is, formally, the closure of the union of the supports of the measures in $\Mmax(\phi)$, though there must necessarily exist a measure in $\Mmax(\phi)$ fully supported on $X$).

Then a finite word satisfies $w\in L(X)$ (\ie it is the prefix of some element $x\in X$) if and only if there exists a periodic word $\omega\in [w]$ for which $\overline{\phi}(\omega)$ -- the (well-defined) average value of $\phi$ over $\omega$ -- satisfies $\overline{\phi}(\omega)=\beta(\phi)\in\mathds{F}$.

\begin{proof}
The $(\Leftarrow)$ implication is direct: if we have $w\in\A^*$ and $\omega\in[w]\subset\A^\N$ a $p$-periodic word satisfying $\overline{\phi}(\omega)=\beta$, then by averaging over the finite $\sigma$-orbit of $\omega$, we obtain the invariant measure $\mu_\omega=\frac{1}{p}\sum_{i=0}^{p-1}\delta_{\sigma^i(\omega)}\in\M_\sigma\left(\A^\N\right)$ such that $\integral{\phi}{\mu_\omega} =\overline{\phi}(\omega)=\beta$, \ie $\mu_\omega\in\Mmax(\phi)$. As $\mu_\omega([w])>0$, it follows that $w$ is in the language of the support of $\mu_\omega$ hence in $X$.

For the $(\Rightarrow)$ implication, we will need to introduce an auxiliary function \emph{cohomologous} to $\phi$. Following the footsteps of Thierry Bousch \cite[Théorème 1, Proposition 10]{Bou01}, we know that $\phi$ is \emph{cohomologous} to a likewise finite-range function $h:\A^r\to\R$ such that $h\leq\beta$.

Given the aforementioned article is written in French, let us give some further clarifications that might otherwise remain lost in translation. Here, cohomologous must be understood as the existence of a continuous map $u\in\Cont\left(\A^\N,\R\right)$ such that $h=\phi+u-u\circ\sigma$. In this case, he obtains a $u$ that satisfies the functional equation $u(x)=\max_{a\in\A}(\phi-\beta + u)(ax)$. Denoting $\tilde{\phi}=\phi-\beta$, we obtain $u(x)=\max_{a_k,\dots,a_1\in\A}\sum_{i=1}^k\tilde{\phi}\left(a_i\dots a_1x\right) + u\left(a_k\dots a_1 x\right)$ by iterating. As $u$ is uniformly continuous on the compact $\A^\N$, for any $\epsilon>0$, we have a rank $k$ such that $\sup_{w\in\A^k,x,y\in\A^\N}\abs{u(wx)-u(wy)}\leq\epsilon$. Furthermore, if $x$ and $y$ share the same first $r-1$ symbols, then the term $\sum_{i=1}^k\tilde{\phi}(\dots)$ is the same in the $\max$ for $u(x)$ and $u(y)$, in which case we can easily deduce that $u(x)=u(y)$ as $k\to\infty$. In other words, we can represent it as $u:\A^{r-1}\to\R$.

As mentionned by Bousch, if we denote $K=h^{-1}(\beta)\subset\A^\N$ (with $h:\A^\N\to\R$), then as $h\leq\beta$, for any (not necessarily invariant) probability measure $\mu$, we have $\integral{h}{\mu}\leq\beta$, with an equality if and only if $h=\beta$ $\mu$-almost-surely, that is if $\supp(\mu)\subset K$. For invariant measures, we can furthermore say that $\mu\in\Mmax(h)$ if and only if $\supp(\mu)\subset Y:=\bigcap_{i\in\N}\sigma^{-i}(K)$. Now, a vertex $w_0$ of the De Bruijn graph $G_r$ satisfies $w_0\in L(Y)$ if and only if there exists an infinite path $\left(w_i\right)_{i\in\N}$ such that for any $i\in\N$, $h\left(w_i\right)=\beta$.

In particular, if we consider $\F=\A^r\setminus h^{-1}(\beta)$ (with $h:\A^r\to\R$), we obtain a set of forbidden patterns such that $Y=X_\F$ is precisely the SFT induced by $\F$ (and then by removing vertices/words that fall outside of a strongly connected component we would obtain $X$ itself). From thereon, it would be expected that we compute $h$ directly, and then define/describe $X$ by restricting $G_r$ to strongly connected components of $h$-maximising vertices. However, the main issue with this approach is that $h$ itself is \emph{known to exist} by compactness arguments, but the proof of its existence doesn't provide an algorithmic procedure to compute it exactly, not even as a limit with known speed of convergence.
The end of the proof, and this whole proposition, aim at allowing us to work around this obstruction with a practical argument.

All this context being laid out, consider $w\in L(X)$, so that by ergodic decomposition, we have some \emph{ergodic} measure $\mu\in\Mmax(\phi)$ for which $\mu([w])>0$. By ergodicity, we have a configuration $\omega\in X\subset Y$ for which $\mu=\mu_{\omega}$ (\ie it is the well-defined Cesàro sum of the sequence $\left(\delta_{\sigma^i(\omega)}\right)_{i\in\N}$). In particular, $\mu([w])$ is the (positive) frequency at which $\sigma^i(\omega)\in [w]$. Equivalently, $\omega\in Y$ can be seen as an infinite path $\left(w_i\right)$ on the $h$-maximising vertices of $G_r$. By finiteness of $\A^r$, we can fix two steps $s,t$ such that $s+\abs{w}<t$, $\sigma^s(\omega),\sigma^t(\omega)\in [w]$ and $w_s=w_t$. We then define the $(t-s)$-periodic configuration $\omega'=\left(w_s,w_{s+1},\dots,w_{t-1},w_t=w_s,\dots\right)$. Now, $w$ has a positive frequency in $\omega'$ (at least $\frac{1}{t-s}$), so $\mu_{\omega'}([w])>0$ still, and $w$ is a prefix of $\omega'$, \ie $\omega'\in[w]$.

The final nail in the coffin is that, by construction:
\begin{itemize}
\item $\omega'$ is periodic, so that $\overline{h}\left(\omega'\right)=\integral{h}{\mu_{\omega'}}$ and $\overline{\phi}\left(\omega'\right)=\integral{\phi}{\mu_{\omega'}}$ are both well-defined,
\item $\omega'$ only visits $h$-maximising vertices of $G_r$, so $\overline{h}\left(\omega'\right)=\beta$,
\item $\phi$ and $h$ are cohomologous, so $\integral{\phi}{\mu_{\omega'}}=\integral{h}{\mu_{\omega'}}$, 
\end{itemize}
so we know that $\overline{\phi}\left(\omega'\right)=\beta$, which concludes the proof.
\end{proof}
\end{proposition}

\begin{corollary}
Let us denote $n=\abs{\A^r}=\abs{\A}^r$. For $w\in\A^r$, define $F\left(w\right)$ as the maximum of $\overline{\phi}(\omega)$ among all the cycles $\left(w_i\right)$ on $G_r$ of period $p\leq n$ with $w_0=w$. Then we have:
\begin{itemize}
\item $F(w)\leq\beta$,
\item $F(w)=\beta$ if and only if $w\in L(X)$, 
\item $X=X_\F$ with $\F=\A^r\setminus F^{-1}(\beta)$.
\end{itemize}

\begin{proof}
The first item follows from the fact that \emph{any} cycle $\omega$ satisfies $\overline{\phi}(\omega)=\integral{\phi}{\mu_\omega}\leq\beta$.

For the second item, clearly $F(w)=\beta$ gives the existence of a cycle $\omega$ that guarantees $w\in L(X)$ in the previous proposition. Conversely, $w\in L(X)$ guarantees the existence of a $h$-maximising $p$-periodic cycle $\omega=\left(w_i\right)$, but it might be \emph{too long}. However, if $p>n$, then $\omega$ contains a repetition, some ranks $0\leq i<j<p$ such that $w_i=w_j$ and so, by skipping all that happens in-between, $\left(w_0,\dots,w_{i-1},w_i=w_j,w_{j+1},\dots,w_{p-1}\right)$ gives a shorter $h$-maximising cycle such that $w_0=w$ (notably, this property breaks for a general $w\in L(X)$: if it is too long to be uniquely determined by $w_0$, by shortening the cycle, we might break the $\omega\in[w]$ property).
By induction, we can always obtain a short-enough period $p\leq n$, so that $F(w)=\overline{\phi}(\omega)=\beta$.

For the third item, we simply need to remark that $F^{-1}(\beta)\subset\A^r$ is precisely the set $L(X)\cap \A^r$, made of all the vertices of $G_r$ that belong to some strongly connected $h$-maximising component, so $X=X_\F$ directly follows.
\end{proof}
\end{corollary}

Consequently, $X$ can be characterised directly through an algorithmic study of $\phi$ itself, without involving any knowledge about the cohomologous function $h=\phi+u-u\circ\sigma$.

\begin{remark}
The previous proposition and corollary both consider a full shift $\left(\A^\N,\sigma\right)$ as the underlying dynamical system. One might wonder what happens if we consider instead an SFT $X_\G$ (defined through $\A$ and $\G\subset\A^*$) to maximise the potential among measures in $\M_\sigma\left(X_\G\right)$.

Without loss of generality, by adding ``padding'' symbols to the elements of $\G$ and/or $\phi$, we might assume that $\G\subset\A^r$ and $\phi:\A^r\to\mathds{F}$.

Then the previous results apply likewise by considering instead periodic words defined on the restriction of the De Bruijn graph $G_r$ \emph{induced} on the subset of vertices $\A^r\setminus\G$.
\end{remark}

\subsection{Explicit Algorithms for the Maximum Mean Weight Cycles}

From the previous analysis, it follows that in the case of a finite-range potential $\phi:\A^r\to\mathds{F}$, computing $\beta$ and $\F$ amounts to computing the average weights of cycles of a directed graph with the vertices $\A^r$.

In the case of a strongly connected directed graph $G=(V,E)$ with weights on the \emph{edges}, Karp's algorithm \cite{Kar78} is known to compute the maximal mean weight of a cycle in $O(\abs{V}\times\abs{E})$ steps (the bound is optimal).
Define $H_k(u,v)$ as the maximum weight of a path of length $k$ from $u$ to $v$ in $G$, and fix an arbitrary source vertex $s\in V$.
The algorithm proceeds by first iteratively computing the families $\left( H_k(s,v)\right)_{v\in V}$ for increasingly large values of $k\leq\abs{V}$ (each step has a cost $\abs{E}$, by testing all the edges one by one to lengthen current maximal paths), and then returns $\beta =\max_{v\in V}\min_{0\leq k <\abs{V}}\frac{H_{\abs{V}}(s,v)-H_k(s,v)}{\abs{V}-k}$ (this final $\max\min$ computation takes $O\left(\abs{V}^2\right)$ steps, and as $G$ is strongly connected, we know that $\abs{E}\geq\abs{V}$).

By giving the start weight $\phi(u)$ to each edge $(u,v)$ (or the end weight $\phi(v)$, which doesn't change the weight of cycles) on the De Bruijn graph $G_r$, using Karp's algorithm computes $\beta$ in $O\left(\abs{\A}^{2r+1}\right)$ steps of computation.

To describe $\F$, the main limitation of Karp's algorithm is that it sort of stumbles upon $\beta$ by happenstance, by finding (at least) one maximising cycle (in a not-so-obvious way, missed by Karp in his article offering a simpler but incorrect argument, and which took four decades to be noticed \cite{ChaCon17}), but without ever reaching a full knowledge of which vertices and cycles realise that maximum, and we could not find appropriate algorithms in the existing literature.

It might be enough to repeatedly run Karp's algorithm for all potential sources $s\in V$ in order to identify all the relevant vertices, but this is not proven and rather than doing this we might as well adapt the algorithm itself directly.
Hence, to compute $F$, taking inspiration from Karp's algorithm, we propose the following procedure. First, compute $\left(H_k(u,v)\right)_{u,v\in V}$ inductively for increasingly large values of $k\leq\abs{V}$: we can initialise $H_0=0$, and given $H_k$, start with $h_{k+1}=-\infty$ and iterate over $u\in V$ and $e=(x,v)\in E$ to replace the current value of $h_{k+1}(u,v)$ by $H_k(u,x)+\phi(x,v)$ if it is bigger, so that $h_{k+1}=H_{k+1}$ when the iteration stops (\ie computing $H_{k+1}$ takes $O(\abs{V}\times\abs{E})$ steps once $H_k$ is known). Finally, simply have $F(u)=\max_{1\leq k\leq\abs{V}}\frac{H_k(u,u)}{k}$. This algorithm allows us to compute not only the mean max weight but also which vertices can realise it, in $O\left(\abs{V}^2\abs{E}\right)$ steps of computation.

Once reframed in our setting (by shifting the weights from the vertices to the edges), we can compute the function $F:\A^r\to\mathds{F}$ (and thus deduce $\F$ from the maximising set of $F$) in $O\left(\abs{\A}^{3r+1}\right)$ steps. This algorithm has been implemented and tested in both Python and Rust, and the code is available on a public repository \cite{Gay26}.

\begin{remark}
As mentioned, Karp's algorithm is optimal for arbitrary graphs. While this is beyond the scope of the paper, we conjecture here that for similar reasons, the $O\left(\abs{V}^2\abs{E}\right)$ bound we obtain is optimal to produce all the vertices that maximise the mean cycle weight.

However, in both cases, this doesn't mean that there is no room for (algorithmic) improvements for our specific question, either because unlike in the aforementioned algorithms, our weights are on the vertices rather than the edges, or because there are specific optimisations to be found for the De Bruijn graphs.
\end{remark}

\sloppy
\hbadness=10000
\printbibliography

\end{document}